\documentclass[11pt]{amsart} 

\input xypic

\usepackage{verbatim}
\usepackage{rotating}
\usepackage{amssymb}
\usepackage{amsmath}

\newtheorem{theorem}{Theorem}
\newtheorem{Theorem}{Theorem}[section]
\newtheorem{proposition}[Theorem]{Proposition}
\newtheorem{corollary}[Theorem]{Corollary}
\newtheorem{lemma}[Theorem]{Lemma}

\theoremstyle{definition}
\newtheorem{definition}[Theorem]{Definition}

\DeclareMathOperator{\rank}{\text{\upshape rank}}
\DeclareMathOperator{\im}{{}\textit{im}\,{}}
\newcommand{\wedgebar}{\raisebox{-3pt}{$\overset{\text{\Tiny\raisebox{2pt}{\text{$\smash{\wedge}$}}}}{\smash{{}-{}}}$}}
\newcommand{\X}{\mathcal{X}}
\newcommand{\Ocal}{\mathcal{O}}
\newcommand{\Ccal}{\mathcal{C}}
\newcommand{\W}{\text{\upshape W}}

\title{On Generalized Whitehead Products}
\author{Brayton Gray}
\address{Department of Mathematics, Statistics and Computer Science\\
         University of Illinois at Chicago\\
         851 S.~Morgan Street\\
         Chicago, IL, 60607-7045, USA} 
\email{brayton@uic.edu}

\begin{document}


\maketitle

Whitehead products have played an important role in
unstable homotopy. They were originally introduced \cite{W}
as a bilinear pairing of homotopy groups:
\[
\pi_m(X)\otimes\pi_n(X)\to \pi_{m+n-1}(X)\quad m,n>1.
\]
This was generalized (\cite{A},\cite{C},\cite{H}) by constructing a map:
\[
\W\colon S(A\wedge B)\to SA\vee SB.
\]
Precomposition with $\W$ defines a function on based homotopy
classes:
\[
[SA,X]\times [SB,X]\to [S(A\wedge B),X]
\]
which is bilinear in case $A$ and $B$ are suspensions.

The case where $A$ and $B$ are Moore spaces was central
to the work of Cohen, Moore and Neisendorfer (\cite{CMN}). In
\cite{A93} and in particular \cite{AG}, this work was generalized.
Much of this has since been simplified
in~\cite{GT}, but further understanding will require 
a generalization from suspensions to co-H spaces.

The purpose of this work is to carry out and study
such a generalization. 
Let $\Ccal\Ocal$ be the category of simply connected co-H spaces
and co-H maps. We define a functor: 
\begin{gather*}
\Ccal\Ocal\times \Ccal\Ocal\to \Ccal \Ocal\\
(G,H)\to G\circ H
\end{gather*}
and a natural transformation:
\begin{equation}\label{eq1}
\W\colon G\circ H\to G\vee H
\end{equation}
generalizing the Whitehead product map. The existence of
$G\circ H$ generalizes a result of Theriault~\cite{T} who showed that 
the smash product of two simply connected co-associative
co-H spaces is the suspension of a co-H space. We do not
need the co-H spaces to be co-associative and require
only one of them to be simply connected.\footnote{In fact we can define
$G\circ H$ for any two co-H spaces
but require at least one of them to be either simply
connected or a suspension in order to obtain the co-H space
structure map on $G\circ H$.} We call $G\circ H$
the Theriault product of $G$ and~$H$.

We summarize our results in the following theorems.
\begin{theorem}\label{theor1}
There is a functor: $\Ccal O\times\Ccal\Ocal\to \Ccal\Ocal$ given by
\[
(G,H)\to G\circ H
\]
and equivalences in $\Ccal O$
\begin{itemize}
\item[(a)]
$(SX)\circ H\simeq X\wedge H$
\item[(b)]
$S(G\circ H)\simeq G\wedge H$
\item[(c)]
$(G_1\vee G_2)\circ H\simeq G_1\circ H\vee G_2\circ H$
\end{itemize}
and homotopy equivalences:
\begin{itemize}
\item[(d)]
$G\circ H\simeq H\circ G$
\item[(e)]
$(G\circ H)\circ K\simeq G\circ(H\circ K)$
\end{itemize}
\end{theorem}
\begin{theorem}\label{theor2}
There is a natural transformation:
\[
\W\colon G\circ H\to G\vee H
\]
which is the Whitehead product map \textup{(\ref{eq1})} in case $G$ and $H$ are
both suspensions. Furthermore, there is a homotopy equivalence:
\[
G\times H\simeq G\vee H\cup_{\W} C(G\circ H).
\]
\end{theorem}

The next theorem concerns the inclusion of the fiber in certain
standard fibration sequences \cite{Gra71}:
\begin{gather*}
G\rtimes \Omega H\xrightarrow{\iota_1} G\vee H\xrightarrow{\pi_2}H\\
\Omega G\ast \Omega H\xrightarrow{\iota_2}G\vee H\xrightarrow{} G\times H.
\end{gather*}
Define $\text{ad}^n(H)(G)$ inductively by $\text{ad}^0(H)(G)=G$ and
\[
\text{ad}^n(H)(G)=[\textit{ad}^{n-1}(H)(G)]\circ H
\]
and an iterated Whitehead product
\[
\textit{ad}^n\colon \textit{ad}^n(H)(G)\to G\vee H
\]
as the composition:
\[
\textit{ad}^n(H)(G)\xrightarrow{\W}\textit{ad}^{n-1}(H)(G)\vee H\xrightarrow{\textit{ad}^{n-1}\vee1} G\vee H\vee H\to G\vee H.
\]
\begin{theorem}\label{theor3}
Suppose $G$ and $H$ are simply connected co-H spaces. Then there
are homotopy equivalences:
\begin{itemize}
\item[(a)]
$G\rtimes \Omega H\simeq \bigvee\limits_{n\geqslant 0}\textit{ad}^n(H)(G)$ where
$\iota_1$ corresponds to $\text{ad}^n$
on the appropriate factor
\item[(b)]
$\Omega G\ast \Omega H\simeq \bigvee\limits_{\substack{i\geqslant 0\\
j\geqslant 1}} \textit{ad}^j(H)(\textit{ad}^i(G)(G))$ where $\iota_2$
corresponds to
$\textit{ad}^j(\textit{ad}^i)$ on to the appropriate factor
\item[(c)]
$S\Omega G\simeq \bigvee\limits_{n\geqslant 0}\textit{ad}^n(G)(G)$ where the
composition
$S\Omega G\to S\Omega G\vee S\Omega G\to G\vee G$ corresponds to the
appropriate iterated Whitehead product on each factor.
\end{itemize}
\end{theorem}
It should be pointed out that the equivalence (c) generalizes
the result of Theriault \cite[1.1]{T} where it is shown that
a simply connected co-associative co-H space
decomposes
\[
\Sigma\Omega G\simeq \bigvee_{n\geqslant 1}M_n
\]
for some spaces $M_n$, which are not further decomposed.
\begin{theorem}\label{theor4}
Suppose $X$ is finite dimensional and $f\!\colon SX\to G\vee H$
then~$f$ is the sum of the projections onto $G$ and $H$ and a
finite sum of iterated Whitehead products.
\end{theorem}

Throughout this work we will assume that all spaces
are of the homotopy type of a CW complex. All homology
and cohomology will be with a field of coefficients. We
will often show that a map between simply connected CW
complexes is a homotopy equivalence by showing
that it induces an isomorphism in homology with
an arbitrary field of coefficients,
without further comment.

Section~\ref{sec1} will be devoted to some general remarks
about telescopes and we will construct the Theriault
product in section~\ref{sec2}. Theorem~\ref{theor1} will follow from
\ref{prop2.3},
\ref{prop2.5} and \ref{prop2.7}. The functor in Theorem~\ref{theor2} is
defined after \ref{cor3.2}
and the equivalence follows from \ref{prop3.8}. The proof of the first part of 
Theorem~\ref{theor3}
occurs just prior to \ref{cor3.5} and ithe rest follows from \ref{cor3.5} 
Theorem~\ref{theor4} follows from \ref{cor3.7}.

\section{}\label{sec1}
In this section we will discuss some general
properties of telescopes of a self may $e\colon G\to G$ where
$G$ is a co-H space. We do not assume that $e$ is idempotent.
We will call $e$ a quasi-idempotent if the induced
homomorphism in homology satisfies the equation:
\[
(e_*)^2=-e_*
\]
where $u$ is a unit. We construct two telescopes:
\begin{align*}
T(e)&\colon G\xrightarrow{e}G\xrightarrow{e}G\xrightarrow{}\dots\\
T(1+e)&\colon G\xrightarrow{1+e}G\xrightarrow{1+e}G\xrightarrow{}\dots
\end{align*}
and a map:
\[
\Gamma\colon G\xrightarrow{} G\vee
G\xrightarrow{\Gamma_1\vee\Gamma_2}T(e)\vee T(1+e).
\]
\begin{proposition}\label{prop1.1}
If $G$ is simply connected and $e$ is a 
quasi-idempotent, $\Gamma$ is a homotopy equivalence.
Furthermore $H_*(T (e))= \im e_*$ and $\widetilde{H}_*(T(1+e))
\cong \ker e_*$.
\end{proposition}
\begin{proof}
Suppose $\Gamma_*(\xi)=0$. Since $(\Gamma_1)_*(\xi)=0$ $(e_*)^k(\xi)=0$
for some $k$, so $e_*(\xi)=0$. Since $(\Gamma_2)_*(\xi)=0$,
$(1+e)_*{}^{k}(\xi)=0$.
But $(1+e)_*^2=(1+e)_*$, so $\xi=-e_*(\xi)=0$. Clearly $\Gamma_*$ is onto. Moreover,
$H_*(T(e))\cong \im e_*$ and $\widetilde{H}_*(T(1+e))\cong
\im(1+e_*)=\ker e_*$.
\end{proof}
\begin{corollary}\label{cor1.2}
Suppose $G$ is a simply connected co-H
space and $A\subset G$ is a retract of~$G$. Let $e$ be the
composition:
\[
G\xrightarrow{r}
A\xrightarrow{i}G.
\]
Then $T(e)\simeq A$, $T(1-e)\simeq G/A$, and the identity map
of $T(e)$ can be factored:
\[
T(e)\xrightarrow{\xi}
A\xrightarrow{i}
G\xrightarrow{r}
A\xrightarrow{\eta}T(e)
\]
where $\xi$ and $\eta$ are inverse homotopy equivalences.
\end{corollary}
\begin{proof}
The telescope $T(e)$ and $A$ are both simply
connected and there are maps $T(e)\to A$ and $A\to T(e)$
making $A$ a retract of $T(\xi)$, and these maps are homotopy
equivalences. By the Van Kampen theorem $G/A\simeq G\cup CA$
is simply connected. Since $1-e\colon G\to G$ factors through
the projection $\pi\colon G\to G/A$, we can factor the identity
map up to homotopy
\[
G\to G\vee G\xrightarrow{r\vee\pi} A\vee G/A\to G
\]
and hence $G\simeq A\vee G/A$. The factorization of the 
identity map of $T(e)$ is obtained by replacing each
space by a telescope where the three in the 
center are constant.
\end{proof}
Now consider two maps $f_1,f_2\!\colon X\to X$.
\begin{proposition}\label{prop1.3}
$T(f_1f_2)\simeq T(f_2f_1)$.
\end{proposition}
\begin{proof}
We define maps between the telescopes:
\[
\xymatrix{
X\ar@{->}[d]_{f_2}\ar@{->}[r]^{f_1f_2}&X\ar@{->}[d]_{f_2}\ar@{->}[r]^{f_1f_2}&X\ar@{->}[d]_{f_2}\ar@{->}[r]&\cdots\\
X\ar@{->}[d]_{f_1}\ar@{->}[r]^{f_2f_1}&X\ar@{->}[d]_{f_1}\ar@{->}[r]^{f_2f_1}&X\ar@{->}[d]_{f_1}\ar@{->}[r]&\cdots\\
X\ar@{->}[r]^{f_1f_2}&X\ar@{->}[r]^{f_1f_2}&X\ar@{->}[r]&\cdots .
}
\]
The composition is the shift map which is a homotopy
equivalence.
\end{proof}
\section{}\label{sec2}
In this section we will consider 
a pair of co-H spaces in which
at least one is simply connected. Let $G$ and $H$ be two
such co-H spaces with their structure determined
by maps $\nu_1\colon G\to S\Omega G$ and $\nu_2\colon H\to S\Omega H$ each of
which is a right inverse to the respective
evaluation maps (See \cite{Ga}), which we label as $\epsilon_1,\epsilon_2$.
We define self maps of $S(\Omega G\wedge \Omega H)$ as follows:
\begin{align*}
&e_1\colon S(\Omega G\wedge \Omega H)\xrightarrow{\epsilon_1\wedge
1}G\wedge \Omega H\xrightarrow{\nu_1\wedge 1}S(\Omega G\wedge \Omega H)\\
&e_2\colon S(\Omega G\wedge \Omega H)\xrightarrow{1\wedge \epsilon_2}
\Omega G\wedge  H\xrightarrow{1\wedge \nu_2}S(\Omega G\wedge \Omega H);
\end{align*}
here we freely move the suspension coordinate to
wherever it is needed. Clearly $e_1$ and $e_2$ are
idempotents but $e_1e_2$ is not an idempotent;
however it is a quasi-idempotent. In fact by
swapping coordinates, one can see that $S(e_1e_2)$
is homotopic to the negative of the composition:
\[
\overline{e}\colon S^2(\Omega G\wedge \Omega H)\xrightarrow{\epsilon_1\wedge
\epsilon_2}G\wedge H\xrightarrow{\nu_1\wedge \nu_2}S^2(\Omega G\wedge \Omega H).
\]
Since $\overline{e}$ is clearly an idempotent $S(e_1e_2)\circ S(e_1e_2)\sim
\overline{e}^2=\overline{e}
\sim -S(e_1e_2)$, so $(e_1e_2)^2_*=-(e_1e_2)_*$.

Now assuming that one of $G,H$ is simply
connected, it follows that $\Omega G\wedge \Omega H$ is connected,
so $S(\Omega G\wedge \Omega H)$ is simply connected. Consequently
\begin{proposition}\label{prop2.1}
If one of $G$ and $H$ is simply
connected, there is a homotopy equivalence:
\[
S(\Omega G\wedge \Omega H)\simeq T(e_1e_2)\vee T(1+e_1e_2).
\]
\end{proposition}

Let 
\[
\theta \colon S(\Omega G\wedge\Omega H)\to T(e_1e_2)
\]
 be the
projection and 
\[
\psi\colon T(e_1e_2)\to S(\Omega G\wedge \Omega H)
\]
be the unique right inverse to $\theta$ which projects trivially
onto $T(1+e_1e_2)$. These maps determine a co-H
space structure on $T(e_1e_2)$.
\begin{definition}\label{def2.2}
$G\circ H=T(e_1e_2)$.
\end{definition}
\begin{proposition}\label{prop2.3}
Given co-H maps $f\colon G\to G'$ and
$g\colon H\to H'$, there is an induced co-H map
\[
f\circ g\colon G\circ H\to G'\circ H'
\]
making $G\circ H$ a functor of two variables. In
addition there are equivalences of co-H spaces.
\begin{itemize}
\item[(a)]
$SX\circ H\simeq X\wedge H$
\item[(b)]
$S(G\circ H)\simeq G\wedge H$
\item[(c)]
$S^1\circ H\simeq H$
\end{itemize}
and there is a homotopy equivalence $G\circ H\simeq H\circ G$.
\end{proposition}
\begin{proof}
Since $f$ and $g$ are co-H maps, the squares
\[
\xymatrix{
G\ar@{->}[r]^{\nu_1}\ar@{->}[d]_f&S\Omega G\ar@{->}[d]_{S \Omega f}\\
G'\ar@{->}[r]^{\nu_1'}&S\Omega G'
}\qquad
\xymatrix{
H\ar@{->}[r]^{\nu_2}\ar@{->}[d]_g&S\Omega H\ar@{->}[d]_{S \Omega g}\\
H'\ar@{->}[r]^{\nu_2'}&S\Omega H'
}
\]
commute up to homotopy.  It follows that $f$ and $g$
induce maps that commute with $e_1$ and $e_2$ and
hence with the equivalences of~\ref{prop2.1}, $\theta$ and $\psi$.
For part a, observe that the composition $e_2e_1$
factors:
\[
\fontsize{8}{8}\selectfont
\xymatrix{
S(\Omega SX\wedge\Omega
H)\ar@{->}[r]^(0.525){\epsilon_1\wedge
1}&SX\wedge \Omega
H\ar@{->}[r]^(0.425){S\iota\wedge
1}\ar@{->}[dr]_{1\wedge G}& S\Omega
SX\wedge \Omega
H\ar@{->}[r]^{1\wedge \epsilon_2}&\Omega
SX\wedge
H\ar@{->}[r]^(0.4){1\wedge \nu_2}&S(\Omega
SX\wedge \Omega H)\\
&&X\wedge H\ar@{->}[ur]_{L\wedge 1}
}
\]
where $(1\wedge \epsilon)(\epsilon_1\wedge 1)$ is a right universe to
$(1\wedge \nu_2)(\iota\wedge 1)$. 
Thus we can apply~\ref{cor1.2} to see that $SX\circ H\simeq X\wedge H$ with
co-H structure given by the composite $(1\wedge \nu_2)(\iota\wedge 1)$. This
is precisely the co-H structure induced by~$\nu_2$.
Part~b follows since $S(G\circ H)$ is the telescope
of~$\overline{e}$ with co-H structure given by $\nu_1\wedge\nu_2$. Part~c
is a special case of part~a: The last statement
follows directly from~\ref{prop1.3}.

For the associativity
assertion in Theorem~\ref{theor1}, it will
be convenient to describe an alternative definition
of $G\circ H$. For this we assume that $G$ is a retract of a space
$SX$ and $H$ a retract of~$SY$:
\[
G\xrightarrow{\nu_1}SX\xrightarrow{\epsilon_1}G\qquad
H\xrightarrow{\nu_2}SY\xrightarrow{\epsilon_2}H.
\]
We can then replace the telescope in the definition
by the telescope of the composition:
\[
T\colon SX\wedge Y\xrightarrow{1\wedge\epsilon_2}X\wedge
H\xrightarrow{1\wedge \nu_2}SX\wedge
Y\xrightarrow{\epsilon_1\wedge1}G\wedge Y\xrightarrow{\nu_1\wedge
1}SX\wedge Y.
\]
The co-H structures defined by these maps are
equivalent to the structures defined by
\[
\widetilde{\nu_2}\colon
G\xrightarrow{\nu_1}SX\xrightarrow{S\widetilde{\epsilon}_1}S\Omega
X\qquad \widetilde{\nu}_1\colon
H\xrightarrow{\nu_2}SY\xrightarrow{S\Omega\widetilde{\epsilon}_2}S\Omega H
\]
and we have a homotopy commutative ladder:
\[
\fontsize{8}{8}\selectfont
\xymatrix{
SX\wedge Y\ar@{->}[r]^{1\wedge
\epsilon_2}\ar@{->}[d]^{S\widetilde{\epsilon}_1\wedge\widetilde{\epsilon}_2}&X\wedge
H\ar@{->}[r]^{1\wedge \nu_2}\ar@{->}[d]^{\widehat{\epsilon}_1\wedge 1}&SX\wedge
Y\ar@{->}[r]^{\epsilon_1\wedge
1}\ar@{->}[d]^{S\widetilde{\epsilon}_1\wedge\widetilde{\epsilon}_2}&G\wedge
Y\ar@{->}[r]^{\nu_2\wedge1}\ar@{->}[d]^{1\wedge\widetilde{\epsilon}_2}&SX\wedge
Y\ar@{->}[d]^{S\widetilde{\epsilon}_1\wedge\widetilde{\epsilon}_2}\\
S\Omega G\wedge \Omega H\ar@{->}[r]^{1\wedge \epsilon}&\Omega G\wedge
H\ar@{->}[r]^(0.45){1\wedge\widetilde{\nu}_2}&S\Omega G\wedge \Omega
H\ar@{->}[r]^{\epsilon\wedge 1}&G\wedge\Omega
H\ar@{->}[r]^(0.425){\widetilde{\nu}_1\wedge 1}&S\Omega G\wedge\Omega H .
}
\]
Hence we have a commutative diagram
\[
\xymatrix{
SX\wedge
Y\ar@{->}[r]\ar@{->}[d]_{S\widetilde{\epsilon}_1\wedge\widetilde{\epsilon}_2}&T\vee\overline{T}\ar@{->}[d]_{\alpha\vee
\beta}\\
S\Omega G\wedge \Omega H\ar@{->}[r]&G\circ H\vee \text{Tel}(1+e_1e_2) 
}
\]
where $\overline{T}$ is the telescope defined by $1+(\nu_1\wedge
1)(\epsilon_1\wedge 1)(1\wedge \nu_2)(1\wedge \epsilon_2)$.
The map $\alpha\colon T\to G\circ H$ is compatible with the homotopy
equivalence $ST\simeq G\wedge H\simeq S(G\circ H)$ so is itself a
homotopy equivalence. Now choose a right homotopy
inverse for the map $SX\wedge Y\to T$ which projects
trivially to $\overline{T}$. It's composite with
$S\widetilde{\epsilon}_1\wedge\widetilde{\epsilon}_2$ will
then project trivially to $T(1+e_1e_2)$. Hence we have
a homotopy commutative diagram:
\[
\xymatrix{
T\ar@{->}[r]\ar@{->}[d]_{\alpha}
&SX\wedge Y\ar@{->}[r]\ar@{->}[d]_{S\widetilde{\epsilon_1}\wedge
\widetilde{\epsilon}_2}
&T\ar@{->}[d]_{\alpha}\\
G\circ H\ar@{->}[r]&S\Omega G\wedge \Omega H\ar@{->}[r]&G\circ H
}
\]
and consequently the co-H structure on $T$ is compatible
under $\alpha$ with the co-H structure on $G\circ H$. We have
proven
\end{proof}
\begin{proposition}\label{prop2.4}
Suppose $G$ is represented as a retract of
$SX$ and $H$ a retract of $SY$. Then $G\circ H$ is homotopy
equivalent to the telescope $T$ of the composition.
\[
SX\wedge Y\to G\wedge Y\to SX\wedge Y\to X\wedge H\to SX\wedge
Y
\]
as co-H spaces, where the co-H structure on $T$ is
given by the equivalence $SX\wedge Y\simeq T\vee\overline{T}$.
\end{proposition}
\begin{proposition}\label{prop2.5}
$(G_1\vee G_2)\circ H\simeq G_1\circ H\vee G_2\circ H$ as
co-H spaces.
\end{proposition}
\begin{proof}
Write $G_i$ as a retract of $SX_i$, $i=1,2$. Then
$G_1\vee G_2$ is a retract of $S(X_1\vee X_2)$. Thus the telescope
for $(G_1\vee G_2)\circ H$ is at each point the wedge of the
telescopes for $G_1\circ H$ and $
G_2\circ H$.
\end{proof}

At this point we will apply \ref{prop2.4} to prove theorem~\ref{theor1} part e, the associativity formula.
We will make repeated use of
\begin{lemma}\label{lem2.6}
There is a homotopy commutative square
\[
\xymatrix{
S(G\circ H)\ar@{->}[d]_{\simeq}\ar@{->}[r]^{S\psi}&S^2X\wedge Y\ar@{->}[d]^{S\theta}\\
G\wedge H\ar@{->}[r]_{\simeq}\ar@{->}[ur]^{\nu_1\wedge\nu_2}&S(G\circ H).
}
\]
\end{lemma}
\begin{proof}
$G\wedge H$ is a retract of $S^2X\wedge Y$, so we may 
apply \ref{cor1.2}.
\end{proof}
\begin{proposition}\label{prop2.7}
$(G\circ H)\circ K\simeq G\circ (H\circ K)$.
\end{proposition}
\begin{proof}
We suppose that $G$, $H$, and $K$ are presented
by retractions
\begin{gather*}
G\xrightarrow{\psi_1}SX\xrightarrow{\theta_1}G\\
H\xrightarrow{\psi_2}SY\xrightarrow{\theta_2}H\\
K\xrightarrow{\psi_3}SZ\xrightarrow{\theta_3}K
\end{gather*}
and we then construct retractions for $G\circ H$ and $H\circ K$:
\end{proof}
\begin{gather*}
G\circ H\xrightarrow{\psi_3}SX\wedge Y\xrightarrow{\theta_3}G\circ H\\
H\circ K\xrightarrow{\psi_4}SY\wedge Z\xrightarrow{\theta_4}H\circ K.
\end{gather*}
Using these we construct retractions for $G\circ(H\circ K)$ and $(G\circ
H)\circ K$
\begin{gather*}
(G\circ H)\circ K\xrightarrow{\psi} S(X\wedge Y\wedge
Z)\xrightarrow{\theta}(G\circ H)\circ K\\    
G\circ (H\circ K)\xrightarrow{\psi'} S(X\wedge Y\wedge
Z)\xrightarrow{\theta'}(G\circ (H\circ K)).
\end{gather*}
We will show that $\theta'\psi\colon (G\circ H)\circ K\to G\circ (H\circ
K)$ is a
homotopy equivalence. By \ref{lem2.6}, we have homotopy
commutative diagrams:
\[
\xymatrix@C=0.5em{
&(G\!{}\circ{}\! H) \wedge K
\ar@{->}[dr]^{\psi_4\wedge\psi_3}
&&
\ar@{->}[dl]_{\psi_1\wedge\psi_5}G\wedge(H\!{}\circ{}\! K)\ar@{->}[dr]^{\simeq}&\\
S((G\!{}\circ{}\! K)\!{}\circ{}\! K)\ar@{->}[ur]^{\simeq}\ar@{->}[rr]^{S\psi}
&&S^2X\!{}\wedge{}\!Y\!{}\wedge{}\!Z\ar@{->}[rr]^{S\theta'}&
&S(G\!{}\circ{}\! (H\!{}\circ{}\! K)) .
}
\]
Suspending and applying \ref{lem2.6} again we obtain a
homotopy commutative diagram:
\[
\xymatrix{
&\ar@{->}[dl]_{\simeq}G\wedge H\wedge K\ar@{->}[dd]_{\psi_1\wedge
\psi_2\wedge \psi_3}\ar@{->}[dr]^{\simeq}&\\
S(G\circ H)\wedge K\ar@{->}[dr]_{S(\psi_4\wedge\psi_3)}&&SG\wedge (H\circ K)\ar@{->}[dl]^{S(\psi_1\wedge \psi_5)}\\
&S^3(X\wedge Y\wedge Z)&
}
\]
from these diagrams it follows that $S^2(\theta'\psi)$ is
a homotopy equivalence and hence $\theta'\psi$ is as well.

\section{}\label{sec3}
In this section we generalize the clutching construction
\cite[Proposition~1]{G} for fibrations over a suspension to fibrations over
a
co-H space. This allows for the decomposition results in 
theorems~\ref{theor2} and~\ref{theor3}.
\begin{proposition}\label{prop3.1}
Suppose $F\to E\to G$ is a fibration where $G$ is
a co-H space. Then $E/F\simeq G\rtimes F$.
\end{proposition}
\begin{proof}
In the case $G=SX$, we have by \cite[Proposition 1]{G}
\[
E\simeq F\cup_{\theta}(CX) \times F.
\]
So $E/F\simeq SX\rtimes F$. It is easy to construct a map $G\rtimes F\to
E/F$
in general. Consider the sequence of pull backs:
\[
\xymatrix{
F\ar@{=}[r]\ar@{->}[d]&F\ar@{=}[r]\ar@{->}[d]&F\ar@{->}[d]\\ 
E\ar@{->}[d]\ar@{->}[r] &E'\ar@{->}[r]\ar@{->}[d] &E\ar@{->}[d]\\
G\ar@{->}[r]^{\nu}&S\Omega G\ar@{->}[r]^{\epsilon}&G . 
}
\]
Then we consider the composite:
\[
E/F\to E'/F\simeq S\Omega G\rtimes F\to G\rtimes F
\]
where the middle equivalence follows since the base is a
suspension. Showing that the composite is a homotopy
equivalence will take some work.

Since $\nu\epsilon\colon S\Omega G\to S\Omega G$ is an idempotent, we can
decompose $S\Omega G$:
\[
S\Omega G\simeq G\vee G' .
\]
We now observe that we can construct a quasifibration model for
a fibration over a one point union. Suppose we have such a fibration
\[
\xymatrix{
E_A\ar@{->}[r]\ar@{->}[d]&\widetilde{E}\ar@{->}[d]&\ar@{->}[l]E_B\ar@{->}[d]\\ 
A\ar@{->}[r]&A\vee B&\ar@{->}[l]B
}
\]
with pull backs $E_A$ and $E_B$ and fiber $F$. Then we can construct
\[
E_A\cup_F E_B
\]
the union of $E_A$ and $E_B$ with the subspace $F$ identified. Then
\[
\xymatrix{
E_A\cup_FE_B\ar@{->}[r]^{\phi}\ar@{->}[d]&\widetilde{E}\ar@{->}[d]\\ 
A\vee B\ar@{=}[r]&A\vee B 
}
\]
$\phi$ is a homotopy equivalence. In our case $S\Omega G\simeq G\vee G'$
and
$E_G=E$, $E_{G'}=G'\times F$, so
\[
\xymatrix{
E'\simeq E\cup_F G'\times F\ar@{->}[d]\\
G\vee G' 
}
\]
is a quasifibering by \cite[2.10]{DT58}.
On the other hand $E'/F\simeq S\Omega G\rtimes F\simeq G\rtimes F\vee G'
\rtimes F$
while $E\cup_FG'\rtimes F\simeq E/F\vee G'\times F$. Since the map
between
$E'$ and $E\cup_FG'\rtimes F$ is a map over $G\vee G'$ we see that
$E/F\simeq G\times F$.
\end{proof}
\begin{corollary}\label{cor3.2}
$S\Omega G\simeq G\rtimes \Omega G$.
\end{corollary}
\begin{proof}
Apply \ref{prop3.1} to the path space fibration over~$G$.
\end{proof}
\begin{proof}[Proof of theorem~\ref{theor2}] Construction: 
We now describe our generalization of the Whitehead
product. Suppose $G$ and $H$ are co-H spaces and one of
them is simply connected. The Whitehead product:
\[
\W\colon G\circ H\to G\vee H .
\]
is then defined as the composition:
\[
G\circ H\xrightarrow{\psi}S\Omega G\wedge \Omega H\simeq \Omega G*\Omega H\xrightarrow{\omega}G\vee H
\]
where $\omega$ is the inclusion of the fiber in the fibration sequence:
\[
\Omega G*\Omega H\xrightarrow{\omega} G\vee H \to G\times H .
\]
\end{proof}
Clearly $\psi$ and $\omega$ are natural transformations, so $\W$ is
as well. 

Before we prove the homotopy equivalence in theorem~\ref{theor2},
we need to establish some results in theorem~\ref{theor3}. We begin by
constructing maps:
\[
\textit{ad}^n\colon\textit{ad}^n(H)(G)\to G\vee H
\]
inductively. For $n=0$ this is just the inclusion of~$G$ in
$G\vee H$. For $n>0$ we define $\textit{ad}^n$ as the composition:
\begin{multline*}
\textit{ad}^n(H)(G)=\left(\textit{ad}^{n-1}(H)(G)\right)\circ H\\
\to (\textit{ad}^{n-1}(H)(G))\vee H
\to (G\vee H)\vee H=G\vee H .
\end{multline*}
Next we calculate the effect of
$\textit{ad}^n$ in loop space homology:
\[
\Omega(\textit{ad}^n)_*\colon H_*(\Omega(\textit{ad}^n(H)(G)))\to H_*(\Omega(G\vee H))
\]
To do this we need some notation. For each co-H space
$G$, write $\sigma^{-1}\colon\widetilde{H}_r(G)\to \widetilde{H}_{r-1}(\Omega G)$ for the composition:
\[
\widetilde{H}_r(G)\xrightarrow{\nu_+}\widetilde{H}_r(S\Omega
G)\simeq\widetilde{H}_{r-1}(\Omega G) .
\]
Let $\{x_i\}$ be a basis for $\widetilde{H}_*(G)$. Then $H_*(\Omega G)$ is
the tensor
algebra on the classes $\{\sigma^{-1}(x_i)\}$. Given two classes
$x\in \widetilde{H}_r(G)$, $y\in \widehat{H}_s(H)$ we will write
\[
x\circ y\in \widetilde{H}_{r+s-1}(G\circ H)
\]
for the class that corresponds to $x\wedgebar  y\in \widetilde{H}_{r+s}(G\wedge H)$ under
the isomorphism:
\[
\widetilde{H}_{r-1}(G\circ H)\cong \widetilde{H}_r(S(G\circ H))\simeq
\widetilde{H}_r(G\wedge H) .
\]
Then the classes $\{x_i\circ y_j\}$ form a basis for $\widetilde{H}_*(G\circ H)$
where $\{x_i\}$ and $\{y_i\}$ respectively are bases for $\widetilde{H}_*(G)$
and $\widetilde{H}_*(H)$.
\begin{proposition}\label{prop3.3}
$(\Omega W)_*(\sigma^{-1}(x\circ
y))=\pm\left[\sigma^{-1}x,\sigma^{-1}y\right]$
where
\[
(\Omega W)_*\colon H_*(\Omega(G\circ H))\to H_*(\Omega(G\vee H)).
\]
\end{proposition}
\begin{proof}
By lemma~\ref{lem2.6}
\[
\psi_*(x\circ y)=e\wedgebar\sigma^{-1}(x)\wedgebar \sigma^{-1}(y)\in H_*(S\Omega G\wedge \Omega H)
\]
so
\[
(\Omega\psi)_*(\sigma^{-1}(x\circ y))=\sigma^{-1}(x)\wedgebar \sigma^{-1}(y)\in H_*(\Omega G\wedge \Omega H)
\]
regarded as a submodule of $H_*(S\Omega G\wedge \Omega H)$. It now
suffices to evaluate the composition:
\[
\Omega G\wedge \Omega H\to\Omega(S\Omega G\wedge \Omega
H)\xrightarrow{\Omega\xi}\Omega(\Omega G*\Omega H)\to \Omega (G\vee
H) 
\]
where $\xi$ is the standard homotopy equivalence $SX\wedge Y\simeq X*Y$.
\end{proof}
\begin{lemma}\label{lem3.4}
The composition:
\[
\Omega G\wedge \Omega H\to \Omega S(\Omega G\wedge \Omega H)\xrightarrow{\Omega\xi}\Omega(\Omega G *\Omega H)\to \Omega(G\vee H)
\]
carries $\sigma^{-1}(x)\wedgebar \sigma^{-1}(y)\in H_*(\Omega G\wedge
\Omega H)$ to $\pm \left[\sigma^{-1}(x),\sigma^{-1}(y)\right]$.
\end{lemma}
\begin{proof}
We first need to describe the homotopy equivalence
\[
SX\wedge Y\xrightarrow{\xi}X*Y .
\]
Here we write points of the join as $tx+(1-t)y$, $0\leq t\leq 1$,
So $X*Y$ is the quotient of $X\times I\times Y$ given by the
identifications $(x,0,y)\sim (x',0,y)$ and $(x,1,y)\sim (x,1,y')$.
Then $\xi$ is given by the formula:
\[
\xi(t,x,y)=\begin{cases}
(*,1-3t,y)&0\leq 3t\leq 1\\
(x,3t-1,y)&1\leq 3t\leq 2\\
(x,3-3t,*)&2\leq 3t\leq 3 
\end{cases} .
\]
The map $\Omega X*\Omega Y\xrightarrow{\omega}X\vee Y$ is given by
\[
(\omega_1,t,\omega_2)\to \begin{cases}
\omega_1(2t)&0\leq 2t\leq 1\\
\omega_2(2-2t)&1\leq 2t\leq 2
\end{cases} .
\]

Combining these we get
\[
S\Omega X*\Omega Y\xrightarrow{\xi}\Omega X*\Omega Y\xrightarrow{\omega}X\vee Y
\]
with a 6-part formula:
\[
(t,\omega_1,\omega_2)=\begin{cases}
(*,\omega_2(6t))&0\leq 6t\leq 1\\
*&1\leq 6t\leq 2\\
(\omega_1(6t-2),*)&2\leq 6t\leq 3\\
(*,\omega_2(4-6t))&3\leq 6t\leq 4\\
*&4\leq 6t\leq 5\\
(\omega_1(6-6t),*)&5\leq 6t\leq 6
\end{cases}
\]
so the adjoint takes the pair $(\omega_1,\omega_2)$ to
the product of loops $\omega_1^{-1}\omega_2^{-1}\omega_1\omega_2$. The effect of this
on a primitive element is the graded commutator.

Now the iterated circle product $\textit{ad}^n(H)(G)$ has
homology generated by classes of the form
\[
\left(\dots \left(\left(x\circ y_1\right)\circ y_2\right)\dots \circ y_n\right)
\]
where $x\in \widetilde{H}_*(G)$ and $y_i\in \widehat{H}_*(H)$. By \ref{prop3.3}
\[
(\Omega W)_*
\left(\sigma^{-1}\left(\dots\left(x\circ y_1\right)\circ y_2\right)\dots \circ y_n\right)
\]
is $\pm$ the graded commutator
\[
\left[\dots\left[\left[\sigma^{-1}(x),\sigma^{-1}(y_1)\right],\sigma^{-1}(y_2)\right]\dots\sigma^{-1}(y_*)\right]
\]
where the classes $x$ and $y_i$ are thought of as
classes in $\widetilde{H}_*(G\vee H)$.
\end{proof}
\begin{proof}[Proof of theorem~\ref{theor3} part \textup{(a)}:]
Now let $G_*=\widetilde{H}_*(G)$ and $H_*=\widetilde{H}_*(H)$.
Let\linebreak[4]
$L(G_*\oplus\nobreak H_*)$ be the free Lie algebra generated by $G_*$ and
$H_*$,
and $L(H_*)$ the free Lie algebra generated by~$H_*$. Then
Neisendorfer has analyzed the kernel
\[
L(G_*\vee H_*)\to L(H_*)
\]
(\cite[8.7.4]{Nei}). He has shown that this is the free Lie algebra
\[
L\left(\bigoplus_{n\geqslant 0}\textit{ad}^n(H_*)(G_*)\right)
\]
The universal enveloping algebra is thus the tensor
algebra generated by the elements $\textit{ad}^n(H_*)(G_*)$ for  $n\geqslant 0$.
Consequently the fiber of the projection $\Omega(G\vee H)\to \Omega H$ is
\[
\Omega\left(\bigvee_{n\geqslant 0}\textit{ad}^n(H)(G)\right)
\]
and this is homotopy equivalent to $\Omega(G\rtimes \Omega H)$ and the
map
\[
\bigvee_{n\geqslant 0}\textit{ad}^n(H)(G)\to G\vee H
\]
which factors through $G \rtimes \Omega H$ establishes the
homotopy equivalence in theorem~\ref{theor3}.
\end{proof}
\begin{corollary}\label{cor3.5}
\textup{(a)}\hspace*{0.5em}If $G$ is simply connected, $S\Omega G\simeq
\bigvee\limits_{n\geqslant 0}\textit{ad}^n(G)(G)$\linebreak[4]
\textup{(b)}\hspace*{0.5em}if both $G$ and $H$ are simply connected
\[
\Omega G*\Omega H\simeq \bigvee_{\substack{i\geqslant 0\\
j\geqslant 1}}\textit{ad}^j(H)\left(\textit{ad}^i(G)(G)\right) .
\]
\end{corollary}
\begin{proof}
For part (a) apply \ref{cor3.2} and theorem~\ref{theor3} part~a. For part (b), we expand
\begin{align*}
\Omega G*\Omega H
&\simeq(S\Omega G)\wedge \Omega H\\
&\simeq\left[\bigvee_{i\geqslant 0}\left(\textit{ad}^i(G)(G)\right)\right]\wedge
\Omega H\\
&\simeq\bigvee_{i\geqslant 0}\left[\textit{ad}^i(G)(G)\wedge \Omega
H\right]\\
&\simeq\bigvee_{\substack{i\geqslant 0\\
j\geqslant 1}}\textit{ad}^j(H)\left(\textit{ad}^i(G)(G)\right)
\end{align*}
using \ref{prop2.5} and theorem~\ref{theor3} part~a.

Given a Theriault product $P=G_1\circ\dots\circ G_s$ with
some fixed association, let us write $\ell(P)=s$ for the
length of~$P$.
\end{proof}
\begin{Theorem}\label{theor3.6}
Suppose $G$ and $H$ are both simply connected
co-H spaces and $k\geqslant 1$. Then there is a locally finite
collection of iterated Theriault products $\{P_{\alpha}(k)\}$
of length $\ell_{\alpha}$ and iterated Whitehead product maps:
\[
\omega_{\alpha}(k)\colon P_{\alpha}(k)\to G\vee H
\]
such that
\[
\Omega (G\vee H)\simeq
\Omega\left(\bigvee_{\ell_{\alpha}>k}P_{\alpha}(k)\right)\times
\Omega\left(\prod_{\ell_{\alpha}\leqslant k}P_{\alpha}(k)\right)
\]
and the factors of the righthand side are
mapped to the lefthand side by the~$\omega_{\alpha}(k)$.
\end{Theorem}
\begin{proof}
For $k=1$ we use the decomposition:
\[
\Omega(G\vee H)\simeq \Omega (\Omega G*\Omega H)\times \Omega (G\times H)
\]
where $\Omega G*\Omega H$ is a boquet of iterated Theriault products of
length at least~$2$ by~\ref{cor3.5}(b). Now we proceed by
induction on~$k$. Among the finite list of products
$P_{\alpha}(k)$ of length $k+1$, choose one which we label~$P$.
Then
\begin{align*}
\Omega(\bigvee_{\ell_{\alpha}>k} P_{\alpha})&\simeq\Omega(P\vee
\bigvee_{\substack{\ell_{\alpha}>k\\
P_{\alpha}\neq P}}P_{\alpha})\\
&\simeq\Omega P\times \Omega(\bigvee_{\substack{\ell_{\alpha}>k\\
P_{\alpha}\neq P}}P_{\alpha}\rtimes \Omega P).
\end{align*}
The second factor has one less product of length $k+1$.
If we repeat this process once for each $P_{\alpha}(k)$ of length
$k+1$, we obtain:
\[
\Omega(\bigvee_{\ell(P_{\alpha})>k}P_{\alpha})=\Omega P_j\times\dots\times
\Omega P_m\times \Omega\Big(\bigvee_{\ell
(P_{\alpha})>k+1}P_{\alpha}(k+1)\Big) .
\]
Now add the $P_1\dots P_m$ to the list of $P_{\alpha}$ with
$\ell(P_{\alpha})\leq k$ to obtain all $P_{\alpha}(k+1)$ with length $\leq
k+1$.
\end{proof}
\begin{corollary}\label{cor3.7}
Suppose  $X$ is a finite dimensional
co-associative co-H space and $f\!\colon X\to G\vee H$
where $G$ and $H$ are simply connected co-H spaces.
Then~$f$ is a sum of iterated Whitehead products.
\end{corollary}
\begin{proof}
Suppose $\dim X=k$ and $f\!\colon X\to G\vee H$ is given.
Decompose $\Omega(G\vee H)$ as in theorem~\ref{theor3.6}
and note that any product $P_{\alpha}$ of length $k$ is at least
$k$ connected. Consequently the restriction of~$\Omega f$:
\[
(\Omega X)^{k-1}\xrightarrow{\Omega f}\Omega(G\vee H)
\]
factors through the product $\Omega (\prod\limits_{\ell_{\alpha}\leq
k}P_{\alpha}(k))$
and the adjoint:
\[
S\left[(\Omega X)^{k-1}\right]\to G\vee H
\]
is a sum of iterated Whitehead products.
However $f$ is the composition of this map with the
co-H space structure map:
\[
X\to S\left[(\Omega X)^{k-1}\right]
\]
which is a co-H map, so $f$ is such a sum as well.
\end{proof}
\begin{proposition}\label{prop3.8}
If $G$ and $H$ are simply connected,
then there is a homotopy equivalence
\begin{equation*}
\phi\colon (G\vee H)\cup_W C(G\circ H)\to G\times H.
\end{equation*}
\end{proposition}
\begin{proof}
Since the composition
\[
G\circ H\to \Omega G*\Omega H\to G\vee H\to G\times H
\]
is null homotopic, there is an extension
\[
C=(G\vee H)\cup  C(G\circ H)\xrightarrow{\phi}G\times H .
\]
The problem is to show that this map is a homotopy equivalence.
We begin by observing that we can construct a right inverse $\zeta$
to $\Omega \phi$ as the sum of the loops on the inclusions of $G$ and $H$
into $C$:
\[
\Omega G\times \Omega H\xrightarrow{\zeta}\Omega
C\xrightarrow{\Omega\phi}\Omega G\times \Omega H
\]
so $(\Omega\phi)_*\colon H_*(\Omega C)\to H_*(\Omega G\times \Omega H)$ is
an epimorphism.
We will complete the proof by showing that the rank of
$H_k(\Omega C)$ is less than or equal to the rank of $H_k(\Omega G\times
\Omega H)$.
We will need several lemmas.
\end{proof}
\begin{lemma}\label{lem3.9}
Write $\Omega G*\Omega H\simeq G\circ H\vee SQ$. Then the
restriction:
\[
SQ\to \Omega G*\Omega H\xrightarrow{\omega}G\vee H\to C
\]
is  null homotopic
\end{lemma}
\begin{proof}
We first look at the homotopy commutative diagram
\[
\xymatrix{
\Omega G*\Omega H\ar@{->}[d]\ar@{->}[r]^{\omega}&G\vee H\ar@{=}[d]\ar@{->}[r]&G\times
H\ar@{->}[d]^{\pi_2}\\   
G\rtimes \Omega H\ar@{->}[r]&G\vee H\ar@{->}[r]&H .
}
\]
Applying theorem~\ref{theor3}a we see that the composition
\[
SQ\to \Omega G*\Omega H\to G\rtimes \Omega H\to G\vee H
\]
is a sum of maps $\gamma_i$ factoring through $\textit{ad}^i(H)(G)\to G\vee H$ for $i\leqslant 2$. By
induction 
on $i$ we see that
\[
\textit{ad}^i(H)(G)\to G\vee H\to C
\]
is null homotopic for $i\geqslant 1$. This follows since
$\textit{ad}^i$ factors
\[
\textit{ad}^i(H)(G)\to\textit{ad}^{i-1}(H)(G)\vee
H\xrightarrow{\textit{ad}^{i-1}\vee 1} G\vee H\vee H\to G\vee H\to C .
\]

It follows from \ref{lem3.9} that the mapping cone of $\omega$
is homotopy equivalent to $C\vee S^2Q$. Recall that
Ganea proved \cite{Ga} that given a fibration sequence $F\to E\to B$.
One can construct a fibration sequence
\[
F*\Omega B\to E\cup CF\xrightarrow{\pi}B
\]
where $\pi$ pinches the cone on $F$ to a point. Apply this
to the fibration sequence:
\[
\Omega G*\Omega H\xrightarrow{\omega}G\vee H\to G\times H,
\]
to obtain:
\[
\Omega G*\Omega H*\Omega(G*H)\to C\vee S^2Q\xrightarrow{\pi}G\times H .
\]
It is possible that the map
$\pi\vert_{S^2Q}$ is nontrivial. However $\pi\vert_{S^2Q}$ is the sum of
the
projections onto $G$ and $H$, so it factors through $C$ up to
homotopy. Using such a factorization we can construct
a homotopy equivalence:
\[
\Gamma\colon C\vee S^2Q\to C\vee S^2Q
\]
such that $\pi\Gamma\vert_{S^2Q}$ is null homotopic. Replacing $\pi$
with $\pi\Gamma$ does not alter the homotopy type of the fiber of
$\pi$, so we can form the following diagram of fibrations:
\[
\xymatrix{
S^2Q\rtimes \Omega C\ar@{->}[d]\ar@{->}[r]&S^2Q\rtimes\Omega C\ar@{->}[d]\\
\Omega G*\Omega H*\Omega(G*H)\ar@{->}[r]\ar@{->}[d]&C\vee
S^2Q\ar@{->}[r]\ar@{->}[d]^{\pi_1}&G\times H\ar@{=}[d]\\
K\ar@{->}[r]&C\ar@{->}[r]^{\phi}&G\times H.  
}
\]
The lefthand verticle fibration has a cross section since
$\pi_1$ does, hence $K$ is a co-H space and we have a splitting:
\[
\Omega (\Omega G*\Omega H*\Omega (G\times H))\simeq \Omega
(S^2Q\rtimes \Omega C)\times\Omega K .
\]
Each space here is the loops on a co-H space, so the
homologies are all tensor algebras. It follows that for each $i\geqslant 0$
\[
\rank H_i(\Omega G*\Omega H*\Omega(G\times H))\geqslant \rank
H_i(S^2Q\rtimes \Omega C).
\]
We now calculate the Poincar\'{e} series for each of these 
spaces. Suppose $\X(G)=1+gt$ and $\X(H)=1+ht$ 
where $g$ and $h$ are polynomial in~$t$ with positive
integral coefficient. We then have the following consequences:
\begin{align*}
\X(\Omega G)&=1+\dfrac{g}{1-g}\\
\X(\Omega H)&=1+\dfrac{h}{1-h}\\
\X(\Omega G*\Omega H)&=1+\dfrac{ght}{(1-g)(1-h)}\\
\X(G\circ H)&=1+ght\\
\X(SQ)&=1+ght\dfrac{g+h-gh}{(1-g)(1-h)}\\
\X(S^2Q\rtimes \Omega C)&=1+ght^2\dfrac{gth-gh}{(1-g)(1-h)}\cdot\X(\Omega C) .
\end{align*}
On the other hand:
\begin{align*}
\X(\Omega G\times \Omega H)&=1+\dfrac{g+h-gh}{(1-g)(1-h)}\\
\intertext{Consequently}
\X(\Omega G*\Omega H*(\Omega G\times \Omega
H))&=1+\dfrac{ght^2(g+h-gh)}{(1-g)^2(1-h)^2}.
\end{align*}
It follows that $\X(\Omega C)\leq \dfrac{1}{(1-g)(1-h)}=\X(\Omega G\times
\Omega H)$
so $\rank H_i(\Omega C)\leq \rank H_i(\Omega G\times \Omega H)$ for all $i$
and thus $(\Omega\phi)_{\alpha}$ is an isomorphism and hence $\phi$ is
as well.
\end{proof}

\nocite*{}
\bibliographystyle{amsalpha}
\bibliography{Gray-Whitehead}

\end{document}